\documentclass{amsart}
\usepackage[
  left=1.0in,    % adjust these to taste
  right=1.0in,
  top=1.0in,
  bottom=1.0in,
  includefoot    % include footer in margin calculations (optional)
]{geometry}
\usepackage{amsthm,amsfonts,amsmath,amssymb}
\usepackage[colorlinks=true]{hyperref}
\usepackage[msc-links,nobysame]{amsrefs} %backrefs,
\usepackage{cleveref}
\usepackage{tikz-cd}
\usepackage{tikz}
\usetikzlibrary{cd}
\usetikzlibrary[arrows.meta]
\usetikzlibrary{decorations.markings}
\tikzset{degil/.style={
            decoration={markings,
            mark= at position 0.5 with {
                  \node[transform shape] (tempnode) {$\backslash$};
                  }
              },
              postaction={decorate}
}
}

\newcommand*\circled[1]{\tikz[baseline=(char.base)]{
            \node[black,shape=rectangle,draw,inner sep=2pt] (char) {#1};}}

\DeclareMathOperator{\im}{Im}

\DeclareMathOperator{\morp}{End}
\DeclareMathOperator{\homo}{Hom}

\renewcommand{\hom}{\homo}

\DeclareMathOperator{\soc}{soc}

\theoremstyle{plain}
\newtheorem{theorem}{Theorem}[section] 
\newtheorem{lemma}[theorem]{Lemma}
\newtheorem{proposition}[theorem]{Proposition}
\newtheorem{corollary}[theorem]{Corollary}

\theoremstyle{definition}

\newtheorem{example}[theorem]{Example}

\theoremstyle{remark}
\newtheorem{remark}[theorem]{Remark} 
\newtheorem{note}[theorem]{Note}

\begin{document}

\pagenumbering{arabic}

\title{On endoartinian and endonoetherian modules}

\author[Gera]{Theophilus Gera} 
\address{(Gera) Department of Mathematics, Sardar Vallabhbhai National Institute of Technology, Surat, Gujarat, India-395007}
\email{geratheophilus@gmail.com}
\thanks{The first author would like to thank SVNIT Alumni Association (SVNIT-AA) and Department of Education, Government of India for the financial assistance. We are grateful to the referees and the Editor-in-Chief for their helpful comments, which improved the presentation of the paper.} 

\author[Patel]{Manoj Kumar Patel} 
\address{(Patel) Department of Science and Humanities, National Institute of Technology Nagaland, Dimapur, Nagaland, India-797103}
\email{mkpitb@gmail.com}

\author[Gupta]{Ashok Ji Gupta}
\address{(Gupta) Department of Mathematical Sciences, Indian Institute of Technology (B.H.U.), Varanasi, Uttar Pradesh, India-221005}
\email{agupta.apm@itbhu.ac.in}

\begin{abstract}
    We investigate \emph{endoartinian modules}, which satisfy the descending chain condition on endoimages, and establish new characterizations that unify classical and generalized chain conditions. Over commutative rings, endoartinianity coincides with rings satisfying the strongly ACCR* with dim$(R)=0$ and strongly DCCR* conditions. For principally injective rings, the endoartinian and endonoetherian rings are equivalent. Addressing a question of Facchini and Nazemian, we provide a condition under which isoartinian and noetherian rings coincide, and we classify semiprime endoartinian rings as finite products of matrix rings over a division ring. We further show that endoartinianity is equivalent to the K\"{o}the rings over principal ideal rings with central idempotents, and characterize such rings as finite products of artinian uniserial rings.
\end{abstract}



\maketitle

\section{Introduction}

The introduction of the ascending chain condition by Noether in 1921 \cite{noether1921idealtheorie} and the descending chain condition by Artin in 1927 \cite{artin1927zur} significantly advanced the structural theory of rings, building upon the foundational framework laid by Wedderburn. A central result in this context is the Hopkins–Levitzki theorem \citelist{\cite{hopkins1939rings} \cite{levitzki1940rings}}, which states that every artinian ring is noetherian, although the converse does not hold in general.

Varadarajan \cite{varadarajan1992hopfian} established a foundational connection between Hopfian and co-Hopfian modules. This relationship was later extended to a subclass of modules known as \emph{Generalized Fitting modules}, which include both strongly Hopfian and strongly co-Hopfian modules, as introduced by Hmaimou, Kaidi, and Campos in \cite{hmaimou2007generalized}. In the same work, the authors revisited Varadarajan’s result and proved that the equivalence between strongly Hopfian and strongly co-Hopfian rings remains valid within this broader framework—that is, when the ring \( R \) is viewed as a left or right module over itself \cite{hmaimou2007generalized}*{Corollary 3.5}.

Numerous generalizations of chain conditions have been introduced by varying submodule and ideal structures. In particular, Facchini and Nazemian \citelist{\cite{facchini2016modules} \cite{facchini2017artinian} \cite{facchini2019iso}} proposed a framework in 2016 that relaxes the equality requirements inherent in standard chain conditions, thereby revealing a range of new structural properties.

Inspired by these developments, several intermediate finiteness conditions have emerged, lying between classical noetherian/artinian behavior and weaker Hopfian-type notions. Building on the work of Hmaimou et al.; Gouaid, Hamed, and Benhissi \cite{gouaid2020endo-noetherian} revisited generalized Fitting modules in 2020. They introduced the concept of \emph{endonoetherian} modules and rings, situated between isonoetherian and strongly Hopfian structures. This notion had first appeared in a preprint by Kaidi and Campos in 2010 \cite{kaidimodules}.

A natural direction arising from these developments is the exploration of Hopkins–Levitzki-type theorems under broader chain conditions. This motivates our study of endoartinian rings and modules as part of a refined hierarchy that connects classical and generalized notions of finiteness. The diagram below illustrates the ring-theoretic chain condition implications central to our analysis, many of which are also discussed in Salem’s review \cite{gouaid2020endo-noetherianreview}:

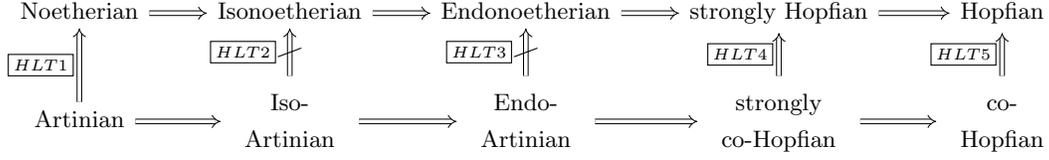
\begin{small}
\begin{center}
    \begin{tikzcd}
        \text{Noetherian} \arrow[r,Rightarrow] & \text{Isonoetherian} \arrow[r,Rightarrow] & \text{Endonoetherian} \arrow[r,Rightarrow] & \text{strongly Hopfian} \arrow[r,Rightarrow] & \text{Hopfian} \\
        \text{Artinian} \arrow[r,Rightarrow] \arrow[u,Rightarrow, "\circled{HLT1} \quad "] & \begin{array}{c}
        \text{Iso-} \\
        \text{Artinian}
    \end{array} \arrow[u, Rightarrow, degil, "\circled{HLT2} \quad "] \ar[r, Rightarrow] & \begin{array}{c}
        \text{Endo-} \\
        \text{Artinian}
    \end{array} \ar[r, Rightarrow] \arrow[u, Rightarrow, degil, "\circled{HLT3} \quad "] & \begin{array}{c}
        \text{strongly} \\
        \text{co-Hopfian}
    \end{array} \ar[r, Rightarrow] \arrow[u, Rightarrow, "\circled{HLT4} \quad"] & \begin{array}{c}
         \text{co-} \\
         \text{Hopfian} 
    \end{array} \arrow[u, Rightarrow, "\circled{HLT5}\quad "]
    \end{tikzcd}
\end{center}
\end{small}

Historically, HLT1 (artinian \( \Rightarrow \) noetherian) was independently proved by Hopkins and Levitzki in the noncommutative setting, building on Akizuki’s 1935 result for commutative rings. Varadarajan \cite{varadarajan1992hopfian}*{Proposition 1.10} later established HLT5, while Hmaimou et al. \cite{hmaimou2007generalized}*{Corollary 3.5} proved HLT4. In contrast, Facchini et al. \cite{facchini2016modules}*{Example 2.6} provided a counterexample to HLT2, showing that isoartinian rings are not necessarily isonoetherian.

Our systematic study of \emph{endoartinian modules} is motivated by their intermediate position between isoartinian and strongly co-Hopfian modules. While classical chain conditions such as artinianity control submodule structure globally, endoartinianity imposes finiteness constraints on the action of the endomorphism ring, capturing internal stabilization phenomena not visible through submodules alone. Such modules often satisfy the descending chain condition on direct summands (see Proposition \ref{endoart.noe.<=>genefitting}) and exhibit strong decomposition behavior. Thus, endoartinianity can enforce internal rigidity, even in the absence of classical finiteness conditions, by stabilizing internal image chains. This structural robustness motivates our investigation into its properties and its connections with other finiteness conditions.  

To deepen this perspective, we systematically investigate the structure of endoartinian modules and derive several characterizations that connect classical and generalized finiteness conditions. Over commutative rings, we show that endoartinianity is equivalent to the rings satisfying strongly ACCR* with dim$(R)=0$ and strongly DCCR* conditions (see Theorem \ref{endoartinian<=>strongly DCCR*<=>strongly ACCR*}). A Hopkins–Levitzki-type theorem is presented in this setting: for right principally injective rings, right endoartinianity and right endonoetherianity coincide (see Theorem \ref{endoartinian <=> endonoetherian}).

In response to a question of Facchini and Nazemian \cite{facchini2017artinian} regarding whether semiprime isoartinian rings are necessarily noetherian, we provide a structural condition under which the answer is indeed positive in the noncommutative case (see Proposition \ref{simple semiprime ring <=> noetherian}).  In an attempt to answer this, we show that endoartinianity and isoartinianity coincide under prime rings (see Corollary \ref{prime endoartinain <=> isoartinian}). We also provide a structural classification of semiprime right endoartinian rings as finite direct products of matrix rings over division rings (see Theorem \ref{semiprime endoartinian <=> f. d.p. matrix rings}).

In \Cref{sec3}, we relate endoartinianity to K\"{o}the rings. We show that for principal ideal rings with central idempotents, endoartinianity is equivalent to being a K\"{o}the ring, and we characterize such rings as finite products of artinian uniserial rings (see Theorem \ref{thm:kothe-endoartinian}). 

Throughout, unless stated otherwise, \( R \) denotes an associative ring with unity, and \( M \) a unital right \( R \)-module. We denote submodules and direct summands by \( \leq \) and \( \leq^{\oplus} \), respectively, and adopt standard notation as in \citelist{\cite{lam1999lectures} \cite{lam2001firstcourse}}.

\section{Properties} \label{sec2}

This section establishes the fundamental properties of endoartinian rings and modules, and provides examples that distinguish them from related concepts. Recall from \cite{kaidimodules} that a (right) module \(M\) is \emph{endoartinian} if every descending chain of endoimages
\[
\im(f_1) \supseteq \im(f_2) \supseteq \cdots
\]
stabilizes, i.e., there exists \( n \in \mathbb{N} \) such that \(\im(f_k)=\im(f_n)\) for all \(k \geq n\), where each \(f_i \in S := \morp_R(M)\) and \(S\) acts on \(M\) from the left. Similarly, a (right) module $M$ is \emph{endonoetherian} if every ascending chain of endokernels stabilizes.

Endoartinian modules strictly generalize isoartinian modules and form a subclass of strongly co-Hopfian modules, as the following examples illustrate.

\begin{example} \label{examples1}
(1) Strongly \(\pi\)-regular modules are endoartinian; see \cite{armendariz1978injective}. 

(2) The \(\mathbb{Z}\)-module \(\mathbb{Q}_{\mathbb{Z}} := \hom_{\mathbb{Z}}(\mathbb{Q}, \mathbb{Q}/\mathbb{Z})\) is endoartinian.

(3) Let $\Bbbk$ be a field and $V=\oplus_{i\geq 1}\Bbbk e_i$, a vector space with countable basis. Let $R\subseteq \morp_{\Bbbk} (V)$ be the ring of all column-finite upper-triangular matrices with respect to the basis $\{e_i\}$. Consider $V$ as a right $R$-module via the natural action. For $n\geq 1$, define $U_n=$ span$\{e_n,e_{n+1},\cdots\}$. Then $V$ is endoartinian as $\morp_R(V)=\Bbbk\cdot $id$_V$, so the only endoimages are $0$ and $V$.

    However, $V$ is \emph{not} isoartinian as $U_1\supsetneq U_2 \supsetneq \cdots$ is an infinite descending chain of submodules. The simple tops $U_n/U_{n+1}$ are 1-dimensional and pairwise non-isomorphic (afforded by distinct characters $\chi_n(r)=r_{nn}$), so $U_n\not\cong U_{m}$ for $n\neq m$.

(4) Let \(R = \prod_{n=1}^\infty \mathbb{F}_2\), the infinite direct product of copies of the field \(\mathbb{F}_2\). Then \(R\) is a commutative Boolean ring and hence von Neumann regular. The right regular module \(R_R\) is \emph{strongly co-Hopfian}: for any \(a \in R\), idempotence implies \(a^n = a\) for all \(n \geq 1\), so the chain \(\im(\lambda_{a^n}) = aR\) stabilizes immediately. However, \(R_R\) is \emph{not} endoartinian: for \(e_k = (0,\dots,0,1,1,\dots)\) (first \(k\) entries zero), the images \(\im(\lambda_{e_k}) = e_kR\) form a strictly descending chain
    \[
    e_1R \supsetneq e_2R \supsetneq e_3R \supsetneq \cdots,
    \]
    so endoartinianity fails. This demonstrates that strong co-Hopfianity does not imply endoartinianity. 
\end{example}

From Example~\ref{examples1}(2), we see that \(\mathbb{Q}_{\mathbb{Z}}\) is endoartinian, whereas \(\mathbb{Z}_{\mathbb{Z}}\) is not even isoartinian, despite \(\mathbb{Z}\) being commutative (see \cite{facchini2016modules}*{Lemma 4.10(1)}). Thus, the endoartinian property is not, in general, inherited by submodules.

\begin{proposition} \label{submodule etc.}
Let \(M\) be a right \(R\)-module.
\begin{enumerate}
    \item If \(N \leq^{\oplus} M\) and \(M\) is endoartinian (respectively, endonoetherian), then \(N\) is endoartinian (respectively, endonoetherian).
    
    \item Suppose \(M\) is both endoartinian and endonoetherian. For every sequence \(\{f_n\}_{n\in\mathbb{N}} \subseteq \morp_R(M)\) satisfying \(f_i f_j = f_{i+j}\) for all \(i,j\in\mathbb{N}\), there exists \(n \in \mathbb{N}\) such that
    \[
    \ker(f_n) \cap \im(f_n) = 0.
    \]
\end{enumerate}
\end{proposition}

\begin{proof}
(1) Write \(M = N \oplus L\). Given \(\{h_n\}\subseteq\morp_R(N)\), extend each \(h_n\) to \(f_n\in\morp_R(M)\) by setting \(f_n(n+\ell):=h_n(n)\) for \(n\in N,\ \ell\in L\) (equivalently, extend \(h_n\) by \(0\) on \(L\)). Then \(\im(f_n)=\im(h_n)\oplus 0\) and \(\ker(f_n)=\ker(h_n)\oplus L\). Hence stabilization of the chains \(\{\im(f_n)\}\) (respectively, \(\{\ker(f_n)\}\)) in \(M\) forces stabilization of \(\{\im(h_n)\}\) (respectively, \(\{\ker(h_n)\}\)) in \(N\). Thus \(N\) is endoartinian (respectively, endonoetherian).

(2) Since \(M\) is both endoartinian and endonoetherian, the descending chain \(\{\im(f_n)\}\) and the ascending chain \(\{\ker(f_n)\}\) both stabilize. Choose \(n\) with \(\im(f_n)=\im(f_{n+1})\) and \(\ker(f_n)=\ker(f_{n+1})\). If \(a\in\ker(f_n)\cap\im(f_n)\), write \(a=f_n(b)\) for some \(b\in M\). Then
    \[
    0=f_n(a)=f_n(f_n(b))=f_{2n}(b),
    \]
    so \(b\in\ker(f_{2n})\). Using the stabilization \(\ker(f_{2n})=\ker(f_n)\), we get \(b\in\ker(f_n)\), hence \(a=f_n(b)=0\). Therefore \(\ker(f_n)\cap\im(f_n)=0\).
\end{proof}

\begin{lemma}[Fitting-type decomposition] \label{Fitting-lemma}
Let \(M\) be both endoartinian and endonoetherian and let \(\{f_n\}_{n \in \mathbb{N}} \subseteq \morp_R(M)\) satisfy \(f_i f_j = f_{i+j}\) for all \(i,j\in\mathbb{N}\). Then there exists \(n \in \mathbb{N}\) such that
\[
M = \ker(f_n) \oplus \im(f_n).
\]
\end{lemma}

\begin{proof}
    Because \(M\) is both endoartinian and endonoetherian, the chains \(\{\im(f_n)\}\) and \(\{\ker(f_n)\}\) stabilize. Choose \(n\) with \(\im(f_n)=\im(f_{n+1})\) and \(\ker(f_n)=\ker(f_{n+1})\). By Proposition~\ref{submodule etc.}(2) we have \(\ker(f_n)\cap\im(f_n)=0\).

    To show \(M=\ker(f_n)+\im(f_n)\), take \(x\in M\). Since \(\im(f_n)=\im(f_{2n})\), there exists \(y\in M\) with \(f_n(x)=f_{2n}(y)=f_n(f_n(y))\). Hence
\[
f_n\big(x - f_n(y)\big)=f_n(x)-f_n(f_n(y))=0,
\]
so \(x-f_n(y)\in\ker(f_n)\). Thus \(x=(x-f_n(y))+f_n(y)\) with \((x-f_n(y))\in\ker(f_n)\) and \(f_n(y)\in\im(f_n)\). Therefore \(M=\ker(f_n)\oplus\im(f_n)\).
\end{proof}

\begin{note}
    A module that is both endoartinian and endonoetherian satisfies the decomposition \( M = \ker(f_n) \oplus \im(f_n) \) for some \( f_n \in \morp_R(M) \) as in Lemma \ref{Fitting-lemma}, and is thus a \emph{strongly generalized Fitting module} in the sense of \cite{hmaimou2007generalized}. 
\end{note}

We now prove that endoartinian (respectively, endonoetherian) modules satisfy a finiteness condition on direct summands, which aligns with one of the defining properties of generalized Fitting modules.

\begin{proposition} \label{endoart.noe.<=>genefitting}
Let \( M \) be an endoartinian (resp., endonoetherian) module. Then \( M \) satisfies the descending (resp., ascending) chain condition on direct summands.
\end{proposition}

\begin{proof}
Let \( M_1 \supseteq M_2 \supseteq \cdots \) be a descending chain of direct summands of \( M \). For each \( i \), let \( f_i \in \morp_R(M) \) be the projection onto \( M_i \). Then the chain 
\[
\im(f_1) \supseteq \im(f_2) \supseteq \cdots
\]
is a descending chain of endoimages of \( M \), which stabilizes since \( M \) is endoartinian. Hence the chain of direct summands stabilizes. The ascending case is dual.
\end{proof}

The following examples show that the classes of endoartinian and endonoetherian modules are not comparable, in general.

\begin{example}
(1) Let \( M = \mathbb{Z}(p^\infty) \), the Prüfer \( p \)-group, viewed as a \( \mathbb{Z} \)-module. Then \( \morp_{\mathbb{Z}}(M) \cong \mathbb{Z}_p \), the ring of \( p \)-adic integers. Every nonzero endomorphism is injective, and the endoimage of the zero map is trivial. Thus, any descending chain of endoimages stabilizes, so \( M \) is endoartinian. However, \( M \) is not endonoetherian as the ascending chain of endokernels
    \[
    \ker(\varphi_n) = \{ m \in M : p^n m = 0 \} \cong \mathbb{Z}/p^n\mathbb{Z},
    \]
    where \( \varphi_n \) denotes multiplication by \( p^n \), does not stabilize.

(2) The module \( \mathbb{Z} \) is endonoetherian, since every ascending chain of endokernels stabilizes. However, it is not endoartinian. To see this, consider the endomorphism \( \varphi_n : \mathbb{Z} \to \mathbb{Z} \) defined by \( \varphi_n(x) = n!x \). Then
    \[
    \im(\varphi_1) \supsetneq \im(\varphi_2) \supsetneq \cdots
    \]
    is a strictly descending chain, since \( \im(\varphi_n) = n!\mathbb{Z} \supsetneq (n+1)!\mathbb{Z} \) for all \( n \).
\end{example}

\begin{proposition} \label{right endoartinian principal}
A ring \(R\) is right endoartinian if and only if for every sequence 
\(\{a_i\} \subseteq R\), there exists \(n \in \mathbb{N}\) such that
\[
a_n R = a_{n+1} R.
\]
Equivalently, there exist \(b, c \in R\) such that
\[
a_n = a_{n+1} b 
\quad \text{and} \quad 
a_{n+1} = a_n c.
\]
\end{proposition}

\begin{proof}
Recall that \(\morp_R(R_R) = \{ \rho_a : a \in R \}\), where 
\(\rho_a(x) = ax\) for all \(x \in R\). A descending chain of 
endoimages corresponds to a chain of principal right ideals:
\[
a_1 R \supseteq a_2 R \supseteq a_3 R \supseteq \cdots.
\]
By definition, \(R\) is right endoartinian if and only if every such 
chain stabilizes, i.e., there exists \(n \in \mathbb{N}\) such that 
\(a_n R = a_{n+1} R\). This equality holds if and only if there exist 
\(b, c \in R\) such that \(a_n = a_{n+1} b\) and \(a_{n+1} = a_n c\).
\end{proof}

The following structural characterization is due to Kaidi and Campos \cite{kaidimodules} and is reproduced here with their kind permission.

\begin{proposition} \label{endonoe <=>acc prin. ann.; endoart <=>dcc on prin. ide.}
Let \( R \) be a ring. Then:
\begin{enumerate}
    \item \( R \) is right endonoetherian if and only if \( R \) satisfies the ascending chain condition on principal right annihilators; that is, for every ascending chain
    \[
    \operatorname{r.ann}(a_1) \subseteq \operatorname{r.ann}(a_2) \subseteq \cdots
    \]
    with \( a_i \in R \), the chain stabilizes.

    \item \( R \) is right endoartinian if and only if \( R \) satisfies the descending chain condition on principal right ideals; equivalently, \( R \) is left perfect \cite{bass1960finitistic}.
\end{enumerate}
\end{proposition}

A ring $R$ is endoartinian (resp. endonoetherian) if it is left and right endoartinian (resp. endonoetherian).

\begin{corollary} \label{subrings}
Let \( R \) be a right endoartinian ring and let \( I \subseteq R \) be a two-sided ideal. Then \( R/I \) is right endoartinian.
\end{corollary}

\begin{proof}
By Proposition \ref{endonoe <=>acc prin. ann.; endoart <=>dcc on prin. ide.}, $R$ is right endoartinian if and only if it is left perfect. Since $R$ is right endoartinian, it is left perfect. Then by \cite{anderson2012rings}*{Corollary 28.7}, the quotient ring $R/I$ is also left perfect. Applying the characterization again, $R/I$ is right endoartinian.
\end{proof}

\begin{remark}
The converse of Corollary \ref{subrings} does not hold in general. For instance, let \( R = \mathbb{Z} \) and consider the ideal \( I = (n) \) for some \( n \geq 2 \). Then \( R/I \cong \mathbb{Z}/n\mathbb{Z} \) is finite, hence artinian and endoartinian. However, \( \mathbb{Z} \) is not semilocal and thus not perfect, so it is not endoartinian. 
\end{remark}

For a submodule $A\leq M$ and a subset $B\subseteq R$, we write $(A:_RB)=\{r\in R \ | \ Br \subseteq A \}$; in particular, $(A:_R x)=\{r\in R \ | \ xr \in A \}$ for $x\in R$. Lu introduced one of the earliest generalizations based on chain conditions for commutative noetherian rings in \cite{lu1988modules}. A module \( M \) is said to satisfy the \emph{ACCR} condition if, for every submodule \( N \leq M \) and ideal \( I \subseteq R \), the ascending chain
\[
(N :_R I) \subseteq (N :_R I^2) \subseteq \cdots
\]
stabilizes. The study of such chain conditions has been extended in various directions. In a recent paper, Gouaid, Hamed, and Benhissi \cite{gouaid2020endo-noetherian} introduced the property (*) for sequences \( \{a_i\}_{i \in \mathbb{N}} \subseteq R \), requiring that the chain
\[
(I :_R a_1) \subseteq (I :_R a_2) \subseteq \cdots
\]
stabilizes for every ideal \( I \subseteq R \). The property (*) implies the \emph{strongly ACCR*} condition (also called property (C)), originally introduced by Visweswaran in \cite{visweswaran1996some}.

On the dual side, Taherizadeh \cite{taherizadeh2002modules} introduced the DCCR condition: a module \( M \) satisfies \emph{DCCR} if, for every submodule \( N \leq M \) and every finitely generated ideal \( I \subseteq R \), the descending chain
\[
NI \supseteq NI^2 \supseteq \cdots
\]
stabilizes. Recently, Naji, {\"O}zen, and Tekir \cite{naji2022strongly} introduced the \emph{strongly DCCR*} condition: for every submodule \( N \leq M \) and every sequence \( \{ a_i \}_{i \in \mathbb{N}} \subseteq R \), the chain
\[
Na_1 \supseteq Na_1a_2 \supseteq Na_1a_2a_3 \supseteq \cdots
\]
stabilizes. They showed that strongly DCCR* lies strictly between artinian and DCCR* conditions.

As discussed earlier for modules, the notion of \emph{strongly DCCR*} extends naturally to rings: a ring \( R \) is said to satisfy the strongly DCCR* condition if, for every ideal \( I \subseteq R \) and every sequence \( \{a_i\}_{i\in \mathbb{N}} \subseteq R \), the descending chain
\[
Ia_1 \supseteq Ia_1a_2 \supseteq Ia_1a_2a_3 \supseteq \cdots
\]
stabilizes. It is straightforward to show that any ring satisfying this property is endoartinian. However, the converse fails unless \( R \) also satisfies the strongly ACCR* condition and has Krull dimension zero; see Theorem \ref{endoartinian<=>strongly DCCR*<=>strongly ACCR*}.

As an example, for any prime \( p \), the \( \mathbb{Z} \)-module \( M = \prod_{n=1}^{\infty} \mathbb{Z}_{p^n} \) is not semi co-Hopfian \cite{aydogdu2008semi}*{Example 2.11}. By \cite{taherizadeh2010characterization}*{Theorem 1.1}, this implies that \( M \) does not satisfy strongly DCCR*, and hence not DCCR*, despite $\mathbb{Z}$ being commutative.

\begin{theorem} \label{endoartinian<=>strongly DCCR*<=>strongly ACCR*}
Let \( R \) be a commutative ring. The following statements are equivalent:
\begin{enumerate}
    \item \( R \) is endoartinian;
    \item \( R \) satisfies the strongly DCCR* condition;
    \item \( R \) satisfies the strongly ACCR* condition and \( \dim(R) = 0 \).
\end{enumerate}
\end{theorem}

\begin{proof}
    Since by Proposition \ref{endonoe <=>acc prin. ann.; endoart <=>dcc on prin. ide.}, being endoartinian, in this case, is same as being perfect, the result follows from \cite{naji2022strongly}*{Corollary 2.8}.
\end{proof}

As an immediate application of Theorem \ref{endoartinian<=>strongly DCCR*<=>strongly ACCR*} and \cite{naji2022strongly}*{Corollary~2.8}, we obtain the following module-theoretic consequence.

\begin{corollary}
Let $R$ be a commutative endoartinian ring and $M$ be an $R$–module. Then every factor module $M/N$ is principally cogenerated.
\end{corollary}

\begin{proof}
Since $R$ is endoartinian, it is perfect by Proposition~\ref{endonoe <=>acc prin. ann.; endoart <=>dcc on prin. ide.} and \cite{naji2022strongly}*{Corollary~2.8} imply that every $R$-module is strongly DCCR$^*$. In particular, $M$ is strongly DCCR$^*$. 

Fix $N\le M$ and $m\in M$, and set $N':=N+Rm$. Apply the strongly DCCR$^*$ property to the constant sequence $a_1=a_2=\cdots=a\in R$. The descending chain
\[
N'a\supseteq N'a^2\supseteq\cdots
\]
stabilizes, so there exists $t\ge1$ with $N'a^t=N'a^{t+1}$. Hence $m a^t\in N' a^{t+1}$, so we may write
\[
m a^t = n a^{t+1} + m b a^{t+1}
\]
for some $n\in N$ and $b\in R$. Rearranging and factoring gives
\[
m a^t(1-ba) = n a^{t+1}\in N.
\]
Put $r:=a^t(1-ba)\in R$. Then $r m\in N$, so the coset $m+N$ is annihilated by the principal ideal $rR$. As $m\in M$ was arbitrary, every coset of $M/N$ has a principal annihilator, i.e.,\ $M/N$ is principally annihilated. By \cite{naji2022strongly}*{Corollary~2.8(5)}, these annihilators can be taken nonzero, so $M/N$ is principally cogenerated.
\end{proof}

The next result is trivial but included for the sake of completion.

\begin{proposition} \label{locali}
    Let $T\subseteq R$ be a multiplicative subset consisting entirely of units. Then, for any right $R$-module $M$, the localization $MT^{-1}:=M\otimes_R RT^{-1}$ is canonically isomorphic to $M$. In particular, $M$ is endoartinian if and only if $MT^{-1}$ is endoartinian.
\end{proposition}

\begin{proof}
    Since every $t\in T$ is a unit, the localization ring $RT^{-1}$ is canonically isomorphic to $R$. The canonical map $\phi: M \to MT^{-1}$ where $\phi(m)=m\otimes 1$ is, therefore, an isomorphism with inverse induced by $R\cong RT^{-1}$. Consequently, $\morp_R(M)\cong \morp_{RT^{-1}}(MT^{-1})$, and the descending chains of endoimages correspond exactly, so $M$ is endoartinian if and only if $MT^{-1}$ is endoartinian.
\end{proof}

The preservation of endoartinianity under localization can fail if the multiplicative set contains nonunits, as shown below.

\begin{example}  \label{counter_localization}
    Let $M=\bigoplus_{n\geq 1}\mathbb{Z}/p\mathbb{Z}$ as a $\mathbb{Z}$-module, and define idempotent endomorphisms  $\pi_k((x_1,x_2,\dots))=(0,\dots,0,x_k,x_{k+1},\dots)$. Then $\im(\pi_1)\supsetneq\im(\pi_2)\supsetneq\cdots$ is a strictly descending chain of endoimages, so $M$ is not endoartinian. 
    
    Now localize at $T=\{1,p,p^2,\dots\}$. Since $p$ acts as zero on each $\mathbb{Z}/p\mathbb{Z}$ summand, we have $MT^{-1} = M\otimes_{\mathbb{Z}} \mathbb{Z}[1/p] = 0$, which is trivially endoartinian. 
\end{example}

The classical Hopkins–Levitzki theorem asserts that artinian rings are also noetherian. In the context of endo-theoretic chain conditions, a natural question arises: under what circumstances does endoartinianity imply endonoetherianity? The previous results—Corollary \ref{subrings}, and Proposition \ref{locali}, show that endoartinianity behaves well under factor rings and localizations. However, this is not sufficient to guarantee endonoetherianity, even for commutative rings. The following example illustrates the failure of such an implication and motivates the introduction of additional structure under which the endoartinian and endonoetherian properties become equivalent.

\begin{example} \label{endoart ring not endonoe}
Let \( R = \Bbbk [x_1, x_2, \dots]/(x_i x_j \mid i,j \in \mathbb{N}) \), where \( \Bbbk \) is a field. In this ring, all products of generators vanish beyond degree one, so \( J(R) = (x_1, x_2, \dots) \) is the Jacobson radical, and \( R \) is local with \( J(R)^2 = 0 \). Hence, \( R \) is perfect and thus endoartinian by Proposition \ref{endonoe <=>acc prin. ann.; endoart <=>dcc on prin. ide.}.

However, \( R \) is not endonoetherian. Define elements \( r_i := 1 + x_i \) and \( s_i := r_1 r_2 \cdots r_i \) for \( i \geq 1 \). It can be shown inductively that
\[
\operatorname{ann}_R(s_n) = (x_{n+1}, x_{n+2}, \dots),
\]
yielding a strictly ascending chain of principal right annihilators:
\[
\operatorname{ann}_R(s_1) \subset \operatorname{ann}_R(s_2) \subset \cdots.
\]
Therefore, \( R \) fails the ascending chain condition on principal right annihilators and is not endonoetherian. 
\end{example}

Example~\ref{endoart ring not endonoe} shows that an endoartinian ring need not be endonoetherian. Conversely, the ring $\mathbb{Z}[x]$ is right endonoetherian but not right endoartinian. Thus, in general, the classes of endoartinian and endonoetherian rings are incomparable. The obstruction here lies in the failure of control over annihilators of principal elements. This is rectified in \emph{right principally injective} rings, where every homomorphism from a principal right ideal extends to a ring by a left multiplication of an element of the ring \cite{nicholson1995principally}. In such settings, descending and ascending chain conditions become symmetric, enabling a Hopkins–Levitzki-type equivalence in the endo-theoretic setting.

\begin{theorem}[Hopkins–Levitzki for Endo-Theory] \label{endoartinian <=> endonoetherian}
Let \( R \) be a right principally injective ring. Then \( R \) is right endoartinian if and only if it is right endonoetherian.
\end{theorem}

\begin{proof}
By \cite{nicholson1995principally}*{Lemma 1.1}, a ring is right principally injective if and only if it satisfies $\operatorname{l.ann}(\operatorname{r.ann}(a))=Ra$ for all $a\in R$. By Proposition \ref{endonoe <=>acc prin. ann.; endoart <=>dcc on prin. ide.}, this equivalence translates precisely to the equivalence between endoartinianity and endonoetherianity.
\end{proof}

\begin{note}
The failure in Example \ref{endoart ring not endonoe} stems from the fact that \( R \) is not principally injective. This can be seen explicitly as follows. Let \( a := x_1 \in R \), and define a homomorphism \( f \colon (a) \to R \) by \( f(x_1) := x_2 \). Suppose, for contradiction, that \( f \) extends to an \( R \)-module homomorphism \( \overline{f} \colon R \to R \), so that the following diagram commutes:

\[
\begin{tikzcd}
0 \arrow[r] & (x_1) \arrow[r, hook, "i"] \arrow[dr, "f"'] & R \arrow[dashed]{d}{\overline{f}} \\
& & R
\end{tikzcd}
\]

Then \( \overline{f}(x_1) = \overline{f}(1 \cdot x_1) = \overline{f}(1) \cdot x_1 = c x_1 \) for some \( c \in R \), while by construction, \( \overline{f}(x_1) = x_2 \). Thus, we would require \( c x_1 = x_2 \), which is impossible in this ring: since all products \( x_i x_j \) vanish, no such \( c \in R \) exists. Hence, the map \( f \) cannot be extended, and \( R \) is not principally injective. 
\end{note}

While the next results are elementary, we include them for completeness. In particular, the following theorem identifies semisimplicity as a structural setting in which endo-chain module conditions coincide with classical ones.

\begin{theorem}
Let \( M \) be a semisimple module. The following statements are equivalent:
\begin{enumerate}
    \item \( M \) is endonoetherian (respectively, noetherian);
    \item \( M \) is endoartinian (respectively, artinian);
    \item \( M \) has finite length.
\end{enumerate}
\end{theorem}

\begin{proof}

(3) \( \Rightarrow \) (1) and (2). If \( M \) has finite length, then it satisfies both the ascending and descending chain conditions on submodules. Since all submodules are direct summands (being semisimple), these chain conditions also hold for endokernels and endoimages. Thus, \( M \) is endonoetherian and endoartinian.

(1) or (2) \( \Rightarrow \) (3). In a semisimple module, every submodule is a direct summand. Therefore, any chain of endokernels or endoimages corresponds to a chain of direct summands. If such a chain terminates, the index set \( I \) must be finite. Hence, \( M \) has finite length.
\end{proof}

Since every module over a semisimple ring is semisimple, the above result allows us to extend the equivalences established in \cite{hmaimou2007generalized}*{Corollary 3.11}. Specifically, two additional chain conditions—endoartinian, and endonoetherian, become equivalent to the previously known eleven conditions.

\begin{theorem} \label{15 cond.}
Let \( R \) be a semisimple ring and \( M \) be an \( R \)-module. Then the following conditions are equivalent:
\begin{enumerate}
    \item \( M \) is noetherian;
    \item \( M \) is strongly Hopfian;
    \item \( M \) is Hopfian;
    \item \( M \) is generalized Hopfian;
    \item Every homogeneous component of \( M \) is finitely generated;
    \item \( M \) has finite length;
    \item \( M \) is artinian;
    \item \( M \) is strongly co-Hopfian;
    \item \( M \) is co-Hopfian;
    \item \( M \) is weakly co-Hopfian;
    \item \( M \) is Dedekind-finite (i.e., \( \morp_R(M) \) is a Dedekind-finite ring);
    \item \( M \) is endonoetherian;
    \item \( M \) is endoartinian.
\end{enumerate}
\end{theorem}

\begin{proof}
Over a semisimple ring, every module is semisimple: it decomposes as a (possibly infinite) direct sum of simple modules, and every submodule is a direct summand. Under these assumptions, the equivalence of conditions (1)–(11) is established in \cite{hmaimou2007generalized}*{Corollary 3.11}.

Now, in the semisimple context, any submodule (and hence any endoimage or endokernel) is a direct summand. Therefore, the ascending and descending chain conditions on submodules coincide with those on endoimages and endokernels. Thus, conditions (12) and (13) are equivalent to condition (6), which asserts that \( M \) has finite length. The equivalence of all thirteen conditions follows.
\end{proof}

\begin{remark}
    The result is generally not true for isoartinian and isonoetherian modules. Facchini and Nazemian gave an example that $\mathcal{L}(M)/\sim$ can be infinite even for a finite length module $M$ where $\mathcal{L}(M)$ is the lattice of submodules of modular lattice with 0 and 1, and $\sim$ smallest congruence in $\mathcal{L}(M)$ \cite{facchini2016modules}*{Example 6.11}. 
\end{remark}

We have seen that, over semisimple rings, many finiteness and rigidity conditions, such as Hopfian, co-Hopfian, and Dedekind-finite properties, coincide with classical chain conditions. One additional structural property of interest in this context is the \emph{Schr\"{o}der–Bernstein property} (or SB property), introduced by Dehghani, Ebrahim, and Rizvi in \cite{dehghani2019schroder}. 

A module \( M \) is said to satisfy the SB property if whenever two direct summands \( A \) and \( B \) of \( M \) are mutually subisomorphic (i.e., there exist embeddings \( A \hookrightarrow B \) and \( B \hookrightarrow A \)), then \( A \cong B \). This condition captures the absence of infinitely nested isomorphic substructures and lies conceptually close to various chain and rigidity conditions.

The following result shows that endoartinian modules automatically satisfy the SB property.

\begin{theorem} \label{endo SB}
Let \( M \) be an endoartinian module. Then \( M \) satisfies the SB property.
\end{theorem}

\begin{proof}
By Proposition \ref{endoart.noe.<=>genefitting}, every endoartinian module satisfies the descending chain condition (DCC) on direct summands. On the other hand, \cite{dehghani2019schroder}*{Theorem 2.14} establishes that any module satisfying DCC on direct summands satisfies the SB property. Therefore, \( M \) also satisfies the SB property.
\end{proof}

In a related development, Dehghani and Rizvi introduced the \emph{dual Schr\"{o}der–Bernstein property} (or, DSB property) in \cite{dehghani2021mutually}. A module \( M \) is said to satisfy the DSB property if every pair of mutually epimorphic modules (i.e., modules \( M_1 \) and \( M_2 \) with epimorphisms \( M_1 \twoheadrightarrow M_2 \) and \( M_2 \twoheadrightarrow M_1 \)) are necessarily isomorphic.

In \cite{dehghani2021mutually}*{Lemma 2.2}, it was shown that the DSB property implies the SB property, but the converse does not hold in general. Given that endoartinian modules satisfy the SB property (\Cref{endo SB}), it is natural to ask whether they also satisfy the DSB property. However, the class of injective modules already provides a counterexample, as shown in \cite{dehghani2021mutually}*{Example 2.3}, demonstrating that satisfying the SB property does not guarantee the DSB property.

\begin{proposition}
    Let $M$ be a right $R$-module. If $M$ is semisimple endoartinian, then $M$ satisfies the DSB property.
\end{proposition}

\begin{proof}
    Since $M$ is semisimple and endoartinian, it has finite length. For finite length modules, mutual epimorphisms imply equal length, hence isomorphism due to semisimplicity.
\end{proof}

Recall that an ideal \( I \subseteq R \) is called \emph{semiprime} if whenever \( J \subseteq R \) is an ideal with \( J^2 \subseteq I \), it follows that \( J \subseteq I \). Similarly, \( I \) is \emph{prime} if for any ideals \( J, K \subseteq R \), the containment \( JK \subseteq I \) implies \( J \subseteq I \) or \( K \subseteq I \). 

The following theorem establishes that, over semiprime rings, the notion of endoartinianity collapses to classical ring-theoretic finiteness. This result appears in Lam's first course \cite{lam2001firstcourse}*{Theorem~10.24}:

\begin{theorem} \label{semiprime endoartinian <=> artinian}
Let \( R \) be a ring. The following conditions are equivalent:
\begin{enumerate}
    \item \( R \) is semiprime and right endoartinian;
    \item \( R \) is semiprime and right artinian;
    \item \( R \) is semisimple.
\end{enumerate}
\end{theorem}

A module \( M \) is called \emph{isosimple} if \( M \neq 0 \) and every nonzero submodule of \( M \) is isomorphic to \( M \). Let \( \mathcal{U} \) denote the class of all isosimple right \( R \)-modules.

The \emph{isosocle} of a right \( R \)-module \( M \) is defined as
\[
\operatorname{I\text{-}soc}(M) := \sum \{ h(U) \mid h \colon U \to M \text{ is an } R\text{-module homomorphism}, \, U \in \mathcal{U} \}.
\]
This notion, introduced as \( \operatorname{Tr}_R(\mathcal{U}) \) in \cite{anderson2012rings}, captures the largest submodule of \( M \) generated by images of isosimple modules. A ring $R$ is said to be isosimple if it isosimple as a right module $R_R$ over itself, that is, it is isomorphic, as a right module, to each of its nonzero right ideals, equivalently, \( R \) is a principal right ideal domain (PRID). These ideas have been explored in depth by Facchini and Nazemian in \citelist{\cite{facchini2016modules} \cite{facchini2017artinian} \cite{facchini2019iso}}.

We now show that for prime rings, the endoartinian and isoartinian conditions are equivalent.

\begin{corollary} \label{prime endoartinain <=> isoartinian}
Let \( R \) be a prime ring. Then \( R \) is right endoartinian if and only if it is right isoartinian.
\end{corollary}

\begin{proof}
Let $R$ be a prime, right endoartinian ring. 
By Theorem~\ref{semiprime endoartinian <=> artinian}, every semiprime right endoartinian ring is right artinian, hence semisimple. Since $R$ is prime and semisimple, it must be simple artinian (that is, $R\cong \mathbb{M}_n(D)$ for some division ring $D$). In a simple artinian ring every right ideal is a finite direct sum of copies of the unique simple module, so any descending chain of right ideals stabilizes up to isomorphism. 
Thus $R$ is right isoartinian.
\end{proof}

Facchini and Nazemian posed the following question in \cite{facchini2017artinian}*{Question 4.11(2)}:
\begin{quote}
    Is a semiprime right isoartinian ring necessarily right noetherian?
\end{quote}

This question was answered affirmatively in the commutative case by Facchini and Nazemian themselves in \cite{facchini2019iso}*{Corollary 3.20}, but remains open in the general noncommutative setting. The next result identifies a structural condition under which the answer is indeed positive in the noncommutative case.

\begin{proposition} \label{simple semiprime ring <=> noetherian}
    Let $R$ be a nonzero simple ring. The following statements are equivalent.
    \begin{enumerate}
        \item $R$ is right isoartinian;
        \item $R\cong \mathbb{M}_n(D)$ for some $n\ge 1$ and some simple principal right ideal domain $D$.
    \end{enumerate}
    Moreover, if, in addition, $D$ is a division ring, then $R$ is simple artinian (hence right artinian and right noetherian).
\end{proposition}

\begin{proof}
    The equivalence between \((1)\) and \((2)\) follows from \cite{facchini2019iso}*{Theorem 3.19(2)}, which shows that a nonzero simple ring is right isoartinian if and only if it is isomorphic to a full matrix ring $\mathbb{M}_n(D)$, where $D$ is a simple principal right ideal domain.

    If, moreover, $D$ is a division ring, then $\mathbb M_n(D)$ is finite-dimensional over $D$ and hence simple artinian. In particular, it is right artinian, and artinian rings are right noetherian (by Hopkins-Levitzki theorem).
\end{proof}

\begin{remark}
The assumption that $R$ is simple is essential in Proposition \ref{simple semiprime ring <=> noetherian}, as it enables the use of both Corollary \ref{prime endoartinain <=> isoartinian} (which requires primeness) and \cite{facchini2019iso}*{Theorem~3.19(2)}.  Outside the simple setting, the structure of right isoartinian rings remains largely unexplored.

Facchini's open problem \cite{facchini2017artinian} — whether every isoartinian ring is necessarily noetherian — remains unresolved, in general. The simple case shows that, within the class of simple rings, a counterexample would necessarily arise from a non-noetherian simple PRID. However, potential counterexamples could also exist among non-simple rings, where the situation is far less understood. 
\end{remark}

We now provide a full structural characterization of semiprime right endoartinian rings. This theorem synthesizes several strands of the theory — endoartinianity and hereditary behavior into a unified classification that parallels the Artin–Wedderburn theorem but in the endo-theoretic context.

\begin{theorem} \label{semiprime endoartinian <=> f. d.p. matrix rings}
Let $R$ be a ring. The following statements are equivalent:
\begin{enumerate}
    \item $R$ is semiprime and right endoartinian;
    \item $R$ is semisimple;
    \item $R \cong \prod_{i=1}^n M_{n_i}(D_i)$, where each $D_i$ is a division ring;
    \item $R$ is right hereditary and soc$(R_R) = R$ is of finite length.
\end{enumerate}
\end{theorem}

\begin{proof}
(1) $\Rightarrow$ (2). By \Cref{semiprime endoartinian <=> artinian}, a semiprime right endoartinian ring is right artinian and semisimple. 

(2) $\Rightarrow$ (3). This is the Wedderburn–Artin theorem: every semisimple ring is isomorphic to a finite product of full matrix algebras over division rings (\cite{lam2001firstcourse}*{Theorem 3.5}).

(3) $\Rightarrow$ (4). Suppose $R\cong \prod_{i=1}^n \mathbb{M}_{n_i} (D_i)$ with each $D_i$ a division ring. Each factor $\mathbb{M}_{n_i} (D_i)$ is simple artinian and decomposes, as a right module over itself, into a direct sum of $n_i$ minimal (hence simple) right ideals. Therefore $R_R$ is a finite direct sum of simple right ideals; equivalently $\soc(R_R)=R$ and the socle has finite length. Moreover, semisimple rings are right (and left) hereditary, since every submodule of a projective module is projective. Hence $R$ is right hereditary. This gives (4).

(4) $\Rightarrow$ (1). Now assume $R$ is right hereditary and $\soc(R_R)=R$ has finite length. From $\soc(R_R)=R$ and the finite-length hypothesis we deduce that $$R_R\cong S_1 \oplus \cdots \oplus S_m$$ for finitely many simple projective right ideals $S_1, \cdots, S_m$; in particular, $R_R$ has finite composition length $m$.

Let $$a_1R\supseteq a_2 R\supseteq \cdots$$ be any descending chain of principal right ideals. Each $a_i R$ is a submodule of $R_R$ and therefore has finite length  $\ell (a_i R)\leq m$. The sequence of nonnegative integers $\ell (a_1 R)\geq \ell (a_2 R)\geq \cdots$ is nonincreasing and bounded below, so it stabilizes. Once the lengths stabilize, inclusion together with equality of lengths forces equality of the submodules: if $X\supseteq Y$ are submodules of a finite length module and $\ell (X)=\ell (Y)$, then $X=Y$. Hence there exists $N$ such that $a_k R=a_N R$ for all $k\geq N$. Thus $R$ satisfies DCC on principal right ideals; by Proposition~\ref{endonoe <=>acc prin. ann.; endoart <=>dcc on prin. ide.} this is equivalent to $R$ being right endoartinian. Moreover, a finite direct sum of simple modules is semiprime. Therefore (1) holds.
\end{proof}

\section{In Relation with K{\"o}the Rings} \label{sec3}

A ring $R$ is called a \emph{right K{\"o}the ring} if each right $R$-module is a direct sum of cyclic right modules (these rings were named to honor G. K{\"o}the who initially studied them in \cite{kothe1935verallgemeinerte}). A ring is \emph{K{\"o}the} if it is both left and right K\"{o}the. A ring is \emph{right duo} if each right ideal is two-sided. K\"{o}the, Cohen and Kaplansky \citelist{\cite{kothe1935verallgemeinerte} \cite{cohen1951rings}} showed that K{\"o}the rings and artinian principal ideal ring coincide if the ring is commutative.

In this section, we explore the structure of rings in which all idempotents are central, under the assumption that the ring is right endoartinian. We establish the equivalence of several important conditions in this context, generalizing classical results on artinian principal ideal rings and K\"{o}the rings. Our aim is to extend \cite{behboodi2014left}*{Corollary 3.3} to the full generality of endoartinian rings, under appropriate structural assumptions.

Our first result is a generalization of a theorem by Habeb \cite{habeb1990note} to the endoartinian setting. It shows how the condition that all idempotents are central forces a ring to decompose into a product of local rings.

\begin{theorem} \label{Habeb theorem}
 Let \( R \) be a right endoartinian ring. The following statements are equivalent:
 \begin{enumerate}
     \item Every idempotent in \( R \) is central;
     \item \( R \cong \prod_{i=1}^n R_i \), where each \( R_i \) is a local ring.
 \end{enumerate}
\end{theorem}

\begin{proof}
(2) \(\Rightarrow\) (1). If \(R\cong\prod_{i=1}^n R_i\) with each \(R_i\) local, the idempotents corresponding to the product decomposition are clearly central. 

(1) \(\Rightarrow\) (2). Assume every idempotent of $R$ is central. Since $R$ is right endoartinian, it is left perfect by Proposition \ref{endonoe <=>acc prin. ann.; endoart <=>dcc on prin. ide.}. In particular, $R$ is semiperfect, so $R/J(R)$ is semisimple artinian and idempotents lift modulo $J(R)$.

Because $R$ is right endoartinian, it satisfies the DCC on principal right ideals. This property implies that $R$ cannot contain an infinite set of orthogonal idempotents. Therefore, there exists a finite complete set of orthogonal primitive idempotents $\{e_1,\cdots, e_n\}\subseteq R$ with $1=e_1+\cdots+e_n$.

For each $i$, the indecomposable projective right ideal $e_i R$ has a local endomorphism ring $e_i Re_i$. Since all idempotents of $R$ are central, by hypothesis, these $e_i$'s  are central. Therefore, the Peirce decomposition gives a ring isomorphism $$R\cong \prod_{i=1}^n e_i Re_i$$ with each factor $e_i Re_i$ local, as required.
\end{proof}

We now present the main theorem of this section, which characterizes K{\"o}the rings in the endoartinian context under the hypotheses that $R$ is a principal ideal ring.

\begin{theorem} \label{thm:kothe-endoartinian}
Let \( R \) be a ring in which every idempotent is central. Assume \( R \) is a principal two-sided ideal ring (i.e., PTIR). Then the following conditions are equivalent:
\begin{enumerate}
  \item \( R \) is endoartinian;
  \item \( R \) is a K\"{o}the ring;
  \item \( R \cong \prod_{i=1}^n R_i \) where each \( R_i \) is an artinian uniserial ring.
\end{enumerate}
\end{theorem}

\begin{proof}
(1) $\Rightarrow$ (3). Assume $R$ is endoartinian. By Proposition \ref{endonoe <=>acc prin. ann.; endoart <=>dcc on prin. ide.}, it is perfect, so it is semiperfect. Since all idempotents in $R$ are central by hypothesis, there exists a finite set of central primitive idempotents $\{e_1, \dots, e_n\}$ such that $1 = e_1 + \cdots + e_n$ and
\[
R \cong e_1Re_1 \times e_2Re_2 \times \cdots \times e_nRe_n 
   = \prod_{i=1}^n R_i.
\]
Each $R_i=e_i Re_i$ is a local artinian ring. Furthermore, since $R$ is a PTIR, each $R_i$ is also a PTIR. In a local artinian ring, the property of being a PTIR forces the lattice of ideals to be linearly ordered; that is, each $R_i$ is artinian and uniserial.

(3) $\Rightarrow$ (1). A finite direct product of artinian rings is artinian. From \cite{behboodi2014left}*{Corollary 3.3}, a finite direct product of uniserial rings is a principal ideal ring. Therefore, $R$ is artinian and a PTIR, which implies it is  endoartinian.

(3) $\Rightarrow$ (2). Assume that $R$ is an artinian uniserial ring and let $M$ be an arbitrary left $R$–module. By Nakayama \cite{nakayama1939frobenius-2}, we can write $M \cong \bigoplus_{i\in \mathbb{Z}} M_i$ as a direct sum of uniserial submodules $M_i$.  By \cite{wisbauer1991foundations}*{55.1(2)}, each finitely generated submodule of any $M_i$ is cyclic. In particular, every finite-length (hence every finitely generated) uniserial summand $M_i$ is cyclic. But each $M_i$ that occurs in the direct-sum decomposition of $M$ has finite length because $R$ is artinian (indeed, each $U_\alpha$ is both uniform and of finite length over an artinian ring). Hence every summand $M_i$ is cyclic. Therefore, $M$ is a direct sum of cyclic modules, proving that $R$ is a K\"{o}the ring.

(2) $\Rightarrow$ (3). If $R$ is a K\"othe ring, then it is  artinian. Since $R$ is also a PTIR, by hypothesis, we may apply the decomposition argument from (1) $\Rightarrow$ (3) to conclude that
\[
R \cong \prod_{i=1}^n R_i,
\]
where each $R_i$ is an artinian uniserial ring.
\end{proof}

The following technical result is useful for analyzing the structure of local rings but is not required for the proofs above. We include it for completeness.

\begin{proposition} \label{endoartinian <=> r/m^2 principal}
    Let \( (R, \mathfrak{m}) \) be a local ring with \( \mathfrak{m}^k = 0 \) for some \( k \geq 1 \). Then \( R \) is a principal right ideal ring if and only if the quotient ring \( R/\mathfrak{m}^2 \) is a principal right ideal ring.
\end{proposition}

\begin{proof}
    \( (\Rightarrow) \). Trivial, as quotients of principal right ideal rings are principal.

    \( (\Leftarrow) \). Assume \( R/\mathfrak{m}^2 \) is a principal right ideal ring. Let \( x_1, \cdots, x_n \in \mathfrak{m} \) be such that their images generate all right ideals of \( R/\mathfrak{m}^2 \). Lifting generators and using Nakayama’s Lemma, we see that each right ideal of \( R \) is generated by one of the lifts up to \( \mathfrak{m}^2 \). An inductive argument on the nilpotency index \( k \) then shows that every right ideal is cyclic. Hence, \( R \) is a principal right ideal ring.
\end{proof}

The hypothesis that all $``$idempotents are central" and $``$PTIR" in Theorem \ref{thm:kothe-endoartinian} is essential and cannot be omitted.

\begin{example} \label{ex:triangular-matrix-counterexample}
Let 
\[
R = \mathbb{T}_2 (\Bbbk) = \left\{ \begin{pmatrix} a & b \\ 0 & d \end{pmatrix} : a,b,d \in \Bbbk \right\}
\]
be the ring of $2\times 2$ upper triangular matrices over a field $\Bbbk$. Since $R$ is finite-dimensional over $\Bbbk$, it is right artinian and hence right endoartinian. Thus, $R$ satisfies condition (1) of Theorem \ref{thm:kothe-endoartinian}.

\begin{itemize}
    \item \emph{Failure of the Equivalent Conditions of of Theorem \ref{thm:kothe-endoartinian}}.
\begin{itemize}
    \item $R$ is {not a Köthe ring} (condition (2)). For instance, the right $R$-module $\Bbbk^2$ (with the natural column vector action) is indecomposable but not cyclic: no single vector generates all of $\Bbbk^2$ under multiplication by upper triangular matrices.
    \item $R$ {does not decompose} as in condition (3). Suppose, for contradiction, that
    \[
    R \cong \prod_{i=1}^n R_i
    \]
    with each $R_i$ an artinian uniserial ring. Then all primitive idempotents in $R$ would be central, contradicting (B). Furthermore, the right ideals $e_{11}R$ and $e_{22}R$ are incomparable, which contradicts the uniserial property that such a decomposition would impose.
\end{itemize}
\end{itemize}

Moreover, $R$ fails the other two equivalent conditions of the theorem, demonstrating that its hypotheses are not superfluous.

\begin{itemize}
    \item \emph{Failure of the PTIR Hypothesis}. Consider the right ideal
\[
I = e_{11}R + e_{12}R, \quad 
e_{11} = \begin{pmatrix} 1 & 0 \\ 0 & 0 \end{pmatrix}, \quad
e_{12} = \begin{pmatrix} 0 & 1 \\ 0 & 0 \end{pmatrix}.
\]
Explicitly,
\[
I = \left\{ \begin{pmatrix} a & b \\ 0 & 0 \end{pmatrix} : a,b \in \Bbbk \right\}.
\]
Suppose, for contradiction, that $I$ is principal, say $I = xR$ for some $x = \begin{pmatrix} \alpha & \beta \\ 0 & \gamma \end{pmatrix} \in R$.  
Then
\[
xR = \left\{ \begin{pmatrix} \alpha r_{11} & \alpha r_{12} + \beta r_{22} \\ 0 & \gamma r_{22} \end{pmatrix} : r_{11},r_{12},r_{22} \in \Bbbk \right\}.
\]
For $xR$ to equal $I$, we must have $\gamma = 0$ and $\alpha \neq 0$. But then $xR$ cannot generate $e_{12}$ independently of $e_{11}$, because every element in $xR$ has top-left entry a multiple of $\alpha$. Hence, no such $x$ exists, and $I$ is not principal. Thus, $R$ is not a principal right ideal ring, so not PTIR.
    \item \emph{Failure of the Central Idempotents Hypothesis}. The idempotent $e = e_{11}$ is not central. Indeed, for $x = e_{12}$,
\[
ex = e_{11} e_{12} = e_{12} \neq 0 = e_{12} e_{11} = xe.
\]
Hence $e$ does not commute with all elements of $R$, showing that not all idempotents in $R$ are central.
\end{itemize}

The ring $\mathbb{T}_2 (\Bbbk)$ demonstrates that a right endoartinian ring can fail both equivalent conditions (2) and (3) when the PTIR and central idempotents hypotheses are not satisfied, confirming that these hypotheses are essential for the theorem's conclusions. While this example does not distinguish which hypothesis is more essential (as both fail), it definitely shows that endoartinianity alone is insufficient and that some additional structural hypotheses are required. 
\end{example}

The results in this section provide a partial answer to the general characterization of noncommutative K\"{o}the rings, an open question initially posed by K\"{o}the \cite{kothe1935verallgemeinerte} and discussed in Tuganbaev's survey \cite{tuganbaev1980rings}. Our theorems characterize K\"{o}the rings within important classes of endoartinian rings.

\bibliography{ref}

\end{document}